\documentclass[10pt,journal]{IEEEtran}
\usepackage{amsmath}           
\usepackage{amssymb}
\usepackage{cite}
\usepackage[latin1]{inputenc}  

\usepackage{amsthm}
\newtheorem{theorem}{Theorem}
\newtheorem{lemma}{Lemma}

\DeclareMathOperator{\E}{E}
\DeclareMathOperator{\Cov}{Cov}

\DeclareMathOperator*{\argmin}{arg\,min}
\DeclareMathOperator*{\argmax}{arg\,max}
\newcommand{\mbf}[1]{\mathbf{#1}}
\newcommand{\mbs}[1]{\boldsymbol{#1}}
\newcommand{\what}[1]{\widehat{#1}}
\newcommand{\wtilde}[1]{\widetilde{#1}}

\newcommand{\FIM}{\mbf{J}}
\newcommand{\LIM}{\mbf{L}}

\newcommand{\pder}{\partial}
\newcommand{\dermat}{\mbf{D}}
\newcommand{\sCov}{\mbs{\Sigma}}
\newcommand{\sMean}{\mbs{\mu}}
\newcommand{\svec}{\mbf{z}}
\newcommand{\pvec}{\mbs{\theta}}
\newcommand{\psub}{\theta}
\newcommand{\y}{\mbf{y}}
\newcommand{\weight}{\mbf{W}}
\newcommand{\T}{\top}
\newcommand{\I}{\mbf{I}}
\newcommand{\0}{\mbf{0}}
\newcommand{\probset}{\mathcal{P}}
\newcommand{\natvec}{\mbs{\lambda}}
\newcommand{\der}[1]{\frac{\partial #1}{\partial \pvec}}
\newcommand{\score}{\der{\ln p_\psub}}
\newcommand{\mscore}{\der{\ln p_\star}}
\newcommand{\dscore}{\der{\ln \delta_\star}}

\begin{document}

\title{Pearson information-based lower bound \\ on Fisher information}
\author{Dave Zachariah and Petre Stoica\thanks{This
    work has been partly supported by the Swedish Research Council
    (VR) under contract 621-2014-5874.}}

\maketitle

\begin{abstract} 
The Fisher information matrix (FIM) plays an important role in the analysis
of parameter inference and system design problems. In a number of
cases, however, the statistical data distribution and its
associated information matrix are either unknown or intractable. For this
reason, it is of interest to develop useful lower bounds on the
FIM. In this lecture note, we derive such a bound based
on moment constraints. We call this bound the Pearson information matrix
(PIM) and relate it to properties of a misspecified data
distribution. Finally, we show that the inverse PIM coincides with the
asymptotic covariance matrix of the optimally weighted generalized method of moments.
\end{abstract}

The  Cramér-Rao bound (CRB) is a useful tool for the analysis of parameter inference
problems, benchmarking
estimators, and system design
\cite{Rao1945_information,Cramer1946_contribution,vanTrees2013_detection,Kay1993_fundamentals}. Let
$\y = [y_1 \cdots y_N ]^\T$ denote the observed data from a system and let
$\pvec \in \mathbb{R}^n$ be the parameters of interest. The CRB exists under
certain regularity conditions and is given by the inverse of the
Fisher information matrix (FIM) $\FIM(\pvec)$, which is a function of
the probability density $p(\y; \pvec)$. More
specifically, the FIM is defined as 
\begin{equation}\label{eq:FIM}
\FIM(\pvec) \triangleq 
\E \left[\der{\ln p(\y; \pvec)} \der{\ln p(\y; \pvec)}^\T \right]
\succeq \0,
\end{equation}
where the gradient $\der{\ln p(\y; \pvec)}$ is known as the `score
function'. Under certain regularity conditions, the score function has
zero mean
\begin{equation}
\E\left[\der{\ln p(\y; \pvec)} \right] = \0.
\label{eq:score_zeromean}
\end{equation}

In many applications, $\FIM(\pvec)$ may not be obtainable. For
instance, $p(\y ; \pvec)$ may be unknown in a practical problem. As an
example, let the observations be modeled as $\y = f(\mbf{u})$, where $f(\cdot)$ is a nonlinear
function and $\mbf{u}$ follows a probability density function
$p(\mbf{u}; \pvec)$. Then a closed-form expression for $p(\mbf{y}; \pvec)$
is not available in general even if $p(\mbf{u}; \pvec)$ is known. 

It is possible, however, to derive tractable lower bounds on the FIM. If $\E[\y]$ is a function of $\pvec$ and
$\Cov[\y]$ is independent of $\pvec$, then using the Gaussian distribution in lieu of
$p(\y; \pvec)$ leads to the minimum FIM $\FIM_G(\pvec) \preceq
\FIM(\pvec)$ (where $G$ stands for Gaussian) and therefore the `worst-case'
inference scenario
\cite{Stoica&Babu2011_gaussian,ParkEtAl2013_gaussian}. A
generalization to $\pvec$-dependent $\Cov[\y]$ was given in
\cite{SteinEtAl2014_lower}. The minimum FIM can be
used for robust system design and estimator formulations. However, in some
cases the minimum information can be overly conservative.

In this lecture note, we derive a tighter lower bound
\begin{equation}\label{eq:basiclowerbound}
\0 \preceq \LIM(\pvec) \preceq \FIM(\pvec) \in \mathbb{R}^{n \times n},
\end{equation}
based on moment constraints. For this reason we call
$\LIM(\pvec)$ the Pearson information matrix (PIM) as an homage to the
inventor of `the method of moments' \cite{Pearson1894_contributions}. As we will see, the PIM is related to
the  generalized method of moments \cite{Hansen1982_large}.

\section{Relevance}

The PIM is a tractable tool for analyzing parameter estimation and system design problems when the statistical data distribution is unknown or intractable.

\section{Prerequisities}

The reader needs basic knowledge about linear algebra,
elementary probability theory, and statistical signal processing.

\section{Preliminaries}

We begin by constructing a function $\svec(\y)$ that contains $M$
statistics of $\y$. We assume that $\svec(\y)$ has computable---either analytically or numerically---mean and covariance
\begin{equation}\label{eq:sMeanCov}
\begin{split}
\sMean(\pvec) &\triangleq \E[\svec(\y)] \in \mathbb{R}^M, \\
\sCov(\pvec)  & \triangleq \E[\left(\svec(\y)
  -\sMean\right)\left(\svec(\y) -\sMean \right)^\T] \in \mathbb{R}^{M \times M},
\end{split}
\end{equation}
where $M \geq n$. For instance, $\svec$ may be constructed using powers
of the data, that is, its elements are made up of empirical moments $\{
y_i \}$ $\{ y_i y_j \}$, $\{ y_i y_j y_k \}$, etc. We assume that $\sCov(\pvec)$  is nonsingular. For notational
simplicity, we drop the argument $\pvec$ in the next analysis and reinstate it when needed. We also write $p_\psub = p(\y  ; \pvec)$.

\section{Pearson information matrix}

In Section~\ref{sec:PIM_algebra}, we begin with a step-by-step algebraic derivation of $\LIM$ in
\eqref{eq:basiclowerbound} which will define the Pearson information
matrix. As explained
there, the PIM generalizes the results in
\cite{Stoica&Babu2011_gaussian,ParkEtAl2013_gaussian,SteinEtAl2014_lower}
and coincides with a bound recently derived in
\cite{SteinEtAl2015_fisher} (by comparison this lecture notes provides
a simple textbook-style derivation of the bound as well as further
connections). In Section~\ref{sec:PIM_connection} we go on to provide an information-theoretic connection
between the PIM and misspecified data distributions using the
principle of maximum entropy \cite{Jaynes1957_information}. Then we
study the behaviour of PIM
when $M$ increases in Section~\ref{sec:PIM_increaseM}.
Finally, in Section~\ref{sec:PIM_achievable}, we establish a relation between the PIM and generalized method of moments that is analogous to the relation between the FIM and the maximum likelihood method. The presented results
enable a tractable analysis of a wider class of data models that
satisfy certain moment constraints.

\subsection{Algebraic derivation}
\label{sec:PIM_algebra}

Consider a linear combination of the centered statistics $\svec-\sMean$:
$$\weight^\T(\svec
- \sMean),$$ where $\weight^\T \in \mathbb{R}^{n \times M}$ denotes a linear
combiner matrix. This vector has zero-mean similar to the score function, cf. \eqref{eq:score_zeromean}. We construct the following matrix
\begin{equation}\label{eq:jointmatrix_init}
\E \begin{bmatrix}
\der{\ln p_\psub} \\
\weight^\T (\svec - \sMean)
\end{bmatrix}
\begin{bmatrix}
\der{\ln p_\psub} \\
\weight^\T (\svec - \sMean)
\end{bmatrix}^\T  \succeq \0.
\end{equation}

Under regularity conditions that allow the interchanging of integral and
derivative operations \cite{Cramer1946_contribution}, the following identity holds:
\begin{equation}\label{eq:momentidentity}
\begin{split}
&\E\left[ \der{\ln p_\psub } (\svec - \sMean)^\T \right] \\
&=  \int
  \frac{1}{p_\psub} \der{p_\psub } (\svec - \sMean)^\T  p_\psub
  d\y  \\
&= \int \left( \der{p_\psub (\svec - \sMean)^\T } + p_\psub\der{\sMean^\T} \right) d\y  \\
&= \frac{\partial}{\partial \pvec}\underbrace{\E\left[ (\svec - \sMean)^\T\right]}_{= \0}  +
\underbrace{\int p_\psub \: d\y}_{=1} \der{\sMean^\T} \\
&= \dermat^\top ,
\end{split}
\end{equation}
where 
\begin{equation}
\dermat^\top = \der{\sMean^\top} \in \mathbb{R}^{n \times M}
\end{equation}
is the gradient of the mean vector. Using \eqref{eq:sMeanCov} and \eqref{eq:momentidentity}, the matrix in \eqref{eq:jointmatrix_init} can be expressed as
\begin{equation}\label{eq:jointmatrix}
\begin{bmatrix}
\FIM & \dermat^\T \weight \\
\weight^\T \dermat & \weight^\T\sCov \weight
\end{bmatrix} \succeq \0.
\end{equation}
It follows from the Schur complement of the lower-right
block of
\eqref{eq:jointmatrix} that
\begin{equation}\label{eq:lowerbound_w}
\0 \preceq \dermat^\T \weight ( \weight^\T \sCov \weight  )^{-1}
\weight^\T \dermat \preceq \FIM,
\end{equation}
assuming that $ \weight^\T \sCov \weight $ has full rank \cite{Horn&Johnson1990_matrix}. Equation
\eqref{eq:lowerbound_w} yields a nonnegative lower bound on the FIM
that is dependent on the choice of the linear combiner $\weight$.

The tightest lower bound
\eqref{eq:lowerbound_w} is found by solving the problem
\begin{equation}\label{eq:opt}
\max_{\weight} \; \dermat^\T \weight ( \weight^\T \sCov \weight  )^{-1} \weight^\T \dermat.
\end{equation}
The combiner that produces the tightest bound is $\weight_\star = \sCov^{-1} \dermat$. To show this, begin by constructing the following positive semidefinite matrix,
\begin{equation}\label{eq:jointmatrix_opt}
\begin{split}
&\quad \begin{bmatrix}
\dermat^\T \sCov^{-1} \dermat & \dermat^\T \weight \\
\weight^\T \dermat & \weight^\T\sCov \weight
\end{bmatrix} \\
&= 
\begin{bmatrix}
\I & \0 \\
\0 & \weight^\T
\end{bmatrix}
\begin{bmatrix}
\dermat^\T \sCov^{-1} \dermat & \dermat^\top \\
\dermat & \sCov
\end{bmatrix}
\begin{bmatrix}
\I & \0 \\
\0 & \weight
\end{bmatrix} \\
&=\begin{bmatrix}
\I & \0 \\
\0 & \weight^\T
\end{bmatrix}
\begin{bmatrix}
\dermat^\T \sCov^{-1/2} \\ 
\sCov^{1/2}
\end{bmatrix}
\begin{bmatrix}
\dermat^\T \sCov^{-1/2} \\ 
\sCov^{1/2}
\end{bmatrix}^\T
\begin{bmatrix}
\I & \0 \\
\0 & \weight
\end{bmatrix}  \succeq \0.
\end{split}
\end{equation}
Using the Schur complement of the lower-right block of
\eqref{eq:jointmatrix_opt} we obtain the upper bound
\begin{equation}\label{eq:bound_opt}
\dermat^\T \weight ( \weight^\T \sCov \weight  )^{-1}
\weight^\T \dermat \preceq \dermat^\T \sCov^{-1} \dermat, 
\end{equation}
which is clearly attained at $\weight_\star = \sCov^{-1} \dermat$.

In conclusion, using \eqref{eq:lowerbound_w} and \eqref{eq:bound_opt},
we have proved the following theorem.
\begin{theorem} The
 optimal lower bound (in the class of bounds considered) is
\begin{equation}\label{eq:lowerbound}
\boxed{\LIM \triangleq \dermat^\T \sCov^{-1} \dermat \preceq \FIM,}
\end{equation}
where $\sCov$ and $\dermat$ are either obtained analytically or computed
numerically. We call $\LIM \succeq \0$ the Pearson information matrix for reasons
explained above.
\end{theorem}

\emph{Remark:} Suppose $\y \sim p_\psub$ can
be modeled as
\begin{equation*}
\y = \sMean(\pvec) + \mbf{w} \in \mathbb{R}^N,
\end{equation*}
where $\mbf{w}$ is a zero-mean random variable. Let $\svec(\y) =
\y$. Then the corresponding PIM coincides with the FIM bounds in
\cite{Stoica&Babu2011_gaussian,ParkEtAl2013_gaussian} and in
\cite{SteinEtAl2014_lower}, when the covariance matrix is fixed,
$\sCov$, and variable, $\sCov(\pvec)$, respectively. The above
algebraic derivation of the PIM provides, moreover, a simple textbook-like proof of the
optimized FIM bound in \cite{SteinEtAl2015_fisher} (which also
contains an illustrative example consisting of a nonlinear
amplification device).

\subsection{Connection to misspecified data distributions}
\label{sec:PIM_connection}

We now relate $\LIM$ to certain properties of misspecified data distributions using the principle of maximum entropy. Instead of the unknown or intractable distribution $p_\psub$, we will use an alternative statistical model, denoted $p_\star$, along with the following identity:
$$\ln p_\psub = \ln p_\star + \ln \delta_\star,$$ 
which holds for any choice of $p_\star$, where $\delta_\star = \frac{p_\psub}{p_\star}$.

The uncertainty of the data $\y$ is quantified by the (differential) entropy which can be decomposed as
\begin{equation*}
\begin{split}
H(p_\psub) &\triangleq -\E[ \ln p_\psub ] = -\E[ \ln p_\star ] - \Delta( p_\psub || p_\star  ),
\end{split}
\end{equation*}
where $\Delta( p_\psub || p_\star  ) = \E[\ln \delta_\star] \geq 0$ is
the divergence of $p_\star$ from the unknown distribution $p_\psub$
\cite{Kullback1959_infostats,Cover&Thomas2012_elements}. We decompose the score function into
\begin{equation}\label{eq:score}
\der{\ln p_\psub} = \mscore  + \dscore,
\end{equation}
where the terms correspond to a misspecified score and a divergence score, respectively. The misspecified information matrix is defined as
\begin{equation}\label{eq:MFIM}
\begin{split}
\FIM_\star &\triangleq \E\left[ \mscore \mscore^\T \right] \succeq \0,
\end{split}
\end{equation}
where the expectation is taken with respect to $p_{\psub}$. 

\begin{lemma}
A general lower bound on $\FIM$ is
\begin{equation}\label{eq:genbound}
\begin{split}
\FIM_\star + \wtilde{\FIM} \preceq \FIM,
\end{split}
\end{equation}
where
\begin{equation}\label{eq:TFIM}
\begin{split}
 \wtilde{\FIM} &= \E\left[ \der{\ln \delta_\star} \der{\ln p_\star   }^\T
\right] + \E\left[ \der{\ln p_\star   } \der{\ln \delta_\star}^\T \right].
\end{split}
\end{equation}
\end{lemma}

\begin{proof}
Inserting \eqref{eq:score} into \eqref{eq:FIM}, we obtain the following decomposition
\begin{equation}\label{eq:FIMdecomp}
\begin{split}
\FIM 
&= \FIM_\star + \wtilde{\FIM} + \E\left[\dscore \dscore^\T\right]
\end{split}
\end{equation}
and the result follows immediately.
\end{proof}

We are concerned with misspecified data models $p_\star$ that
satisfy the given constraint $\E[\svec] = \sMean$. That is, distributions that satisfy
\begin{equation}
\begin{split}
\int \svec p_\star \: d\y = \sMean.
\end{split}
\label{eq:momentconstraint}
\end{equation}
In particular, we let $p_\star$ correspond to the maximum uncertainty of
$\y$. The distribution with the maximum (differential) 
entropy is known to be
\begin{equation*}
\begin{split}
p_\star &\equiv \argmax_{p' \in \probset} \;
-\E'[\ln p'] = \exp( \natvec^\T
\svec - \lambda_0),
\end{split}
\end{equation*}
where $\mathcal{P}$ is the set of valid probability
distributions for $\y$ that satisfy \eqref{eq:momentconstraint} and
$\lambda_0, \natvec$ are multipliers that are chosen to satisfy the
constraint (assuming that the problem is feasible) \cite{Cover&Thomas2012_elements}. For
completeness, we prove this result by noting that the following upper
bound holds for any $p'$:
\begin{equation*}
\begin{split}
H(p') &= - \E'[  \ln p' ]  \\
&\leq - \E'[  \ln p_\star ] \\
&= -\int ( \natvec^\T \svec - \lambda_0) p' \ d\y \\
&= -\int ( \natvec^\T
\svec - \lambda_0) p_\star \ d\y \\
&= H(p_\star).
\end{split}
\end{equation*}
 The equality in the penultimate line follows since both $p'$ and
$p_\star$ satisfy the constraint \eqref{eq:momentconstraint}. The maximum entropy distribution therefore belongs to the exponential
family, that is,
\begin{equation}\label{eq:maxent_dist}
p_\star(\y ; \pvec) = \exp\left( \natvec^\top(\pvec) \svec(\y) - \lambda_0(\pvec) \right),
\end{equation}
where
\begin{equation*}
\lambda_0(\pvec) = \ln \int \exp\left( \natvec^\top(\pvec) \svec(\y) \right) d\y
\end{equation*}
is a normalizing constant.

\begin{lemma}
Using the maximum entropy distribution \eqref{eq:maxent_dist},
the bound in \eqref{eq:genbound} is given by
\begin{equation}\label{eq:gaptight}
\begin{split}
\FIM_\star + \wtilde{\FIM} = \LIM - \left(\sCov^{-1} \dermat - \der{\natvec} \right)^\T \sCov
\left(\sCov^{-1} \dermat - \der{\natvec} \right)
\end{split}
\end{equation}
\end{lemma}

\begin{proof}
For \eqref{eq:maxent_dist}, we have that
\begin{equation}\label{eq:misspecifiedscore_basic}
\begin{split}
\der{\ln p_\star} &=  \der{\natvec^\top} \svec -
\der{\lambda_0} \\
&= \der{\natvec^\top} (\svec - \sMean)
\end{split}
\end{equation}
where $\der{\natvec^\top} $ is $n \times M$. The second
line follows from the fact that 
\begin{equation}\label{eq:grad_logpartition}
\begin{split}
\der{\lambda_0} &= \frac{1}{\int \exp( \natvec^\top \svec) d\y'} \int
\der{\exp( \natvec^\top \svec)} \: d\y\\
&= \int \der{\natvec^\top} \svec \left( \frac{ \exp(
    \natvec^\top \svec)}{\int \exp( \natvec^\top
    \svec ) d\y'} \right) d\y \\
&= \int \der{\natvec^\top} \svec p_\star \:  d\y \\
&= \int \der{\natvec^\top} \svec p_\psub \: d\y \\
&= \der{\natvec^\top} \sMean.
\end{split}
\end{equation}
Next, from the proof of \eqref{eq:genbound} we have that
\begin{equation}\label{eq:gap}
\begin{split}
\FIM
- (\FIM_\star + \wtilde{\FIM}) = \E\left[ \dscore \dscore^\T \right] .
\end{split}
\end{equation}
Using \eqref{eq:score}, the divergence score can be written as the sum
of two random vectors,
\begin{equation}\label{eq:dscore}
\begin{split}
\dscore &= \left(\score - \dermat^\T \sCov^{-1}(\svec - \sMean) \right) \\
&\quad  + \left( \dermat^\T \sCov^{-1}(\svec - \sMean) - \mscore \right).
\end{split}
\end{equation}
For the maximum entropy distribution, these two random vectors are orthogonal, i.e.,
\begin{equation*}
\begin{split}
&\E\Bigl[ \left(\score - \dermat^\T \sCov^{-1}(\svec - \sMean) \right) \\
&\quad \times \left( \dermat^\T \sCov^{-1}(\svec - \sMean) - \mscore \right)^\T \Bigr] \\
&=  \left( \dermat^\top - \dermat^\T \sCov^{-1}\sCov
\right)\left(\sCov^{-1} \dermat - \der{\natvec} \right) \\
&= \0,
\end{split}
\end{equation*}
where the equality follows from \eqref{eq:misspecifiedscore_basic}
and \eqref{eq:momentidentity}. Finally, by inserting
\eqref{eq:dscore} into \eqref{eq:gap}, the right hand side of \eqref{eq:gap} equals
$$\FIM - \dermat^\T \sCov^{-1} \dermat +   \left(\sCov^{-1} \dermat - \der{\natvec} \right)^\T \sCov
\left(\sCov^{-1} \dermat - \der{\natvec} \right)$$
and result \eqref{eq:gaptight} follows.
\end{proof}

\begin{theorem}\label{thm:bound}
The tightest bound \eqref{eq:genbound} is the PIM:
\begin{equation}\label{eq:lowerbound_maxent}
\boxed{\FIM_\star + \wtilde{\FIM} \preceq \LIM \preceq \FIM,}
\end{equation}
and the corresponding misspecified information matrix is $\FIM_\star = \LIM$.
\end{theorem}

\begin{proof}
It follows from \eqref{eq:misspecifiedscore_basic} that for maximum
entropy distributions, the misspecified information matrix is
$$\FIM_\star = \der{\natvec^\T} \sCov \der{\natvec}.$$
Furthermore, in \eqref{eq:gaptight} it is readily seen that the
tightest bound is attained for
\begin{equation}\label{eq:tightcondition}
\der{\natvec^\T}
=\dermat^\T \sCov^{-1}.
\end{equation}
Therefore \eqref{eq:tightcondition} leads to $\FIM_\star = \LIM$ 
and $\wtilde{\FIM} = \0$ in \eqref{eq:genbound}.
\end{proof}

\subsection{The PIM increases as $M$ increases}
\label{sec:PIM_increaseM}

The vector $\svec$ employs $M$ statistics, and to stress that we 
write $\LIM_M = \dermat^\T_M \sCov^{-1}_M \dermat_M$. 
\begin{theorem}
Including more statistics in $\svec$ can never worsen the bound, i.e.,
\begin{equation}
\0 \preceq \LIM_{n} \preceq \cdots \preceq \LIM_{M} \preceq \LIM_{M+1} \preceq \FIM.
\end{equation}
\end{theorem}
\begin{proof}
Suppose we
extend the vector $\svec$ with an $(M+1)$th statistic so that we can write
\begin{equation}
\sCov_{M+1} = \begin{bmatrix}
\sCov_{M} & \mbf{c} \\
\mbf{c}^\T & \kappa
\end{bmatrix} \quad \text{and} \quad
\dermat_{M+1} = \begin{bmatrix}
\dermat_{M} \\
\mbf{d}^\T
\end{bmatrix} .
\end{equation}
Then calculating $\sCov^{-1}_{M+1}$ using
\cite[Lemma~A.2]{Soderstrom&Stoica1988_system}, we obtain
\begin{equation*}
\begin{split}
\LIM_{M+1} &= \begin{bmatrix}
\dermat^\T_M & \mbf{d}
\end{bmatrix}\begin{bmatrix}
\sCov_{M} & \mbf{c} \\
\mbf{c}^\T & \kappa
\end{bmatrix}^{-1}\begin{bmatrix}
\dermat_{M} \\
\mbf{d}^\T
\end{bmatrix} \\
&=  \begin{bmatrix}
\dermat^\T_M & \mbf{d}
\end{bmatrix}
\Bigl(  \begin{bmatrix} \I\\ \0 \end{bmatrix}
  \sCov^{-1}_{M}  \begin{bmatrix} \I & \0 \end{bmatrix}
  \\
&\quad +  \frac{1}{\kappa - \mbf{c}^\T \sCov^{-1}_M \mbf{c}}\begin{bmatrix} -\sCov^{-1}_M \mbf{c} \\
    1 \end{bmatrix}  \begin{bmatrix} -\mbf{c}^\T \sCov^{-1}_M & 1 \end{bmatrix} \Bigr) \begin{bmatrix}
\dermat_M \\ \mbf{d}^\T
\end{bmatrix} \\
&= \LIM_{M} + \frac{(\mbf{d} - \dermat^\T \sCov^{-1}_M
  \mbf{c})(\mbf{d} - \dermat^\T \sCov^{-1}_M \mbf{c})^\T}{\kappa -
  \mbf{c}^\T \sCov^{-1}_M \mbf{c}}\\
&\succeq \LIM_{M}.
\end{split}
\end{equation*}
\end{proof}
An interesting research problem is to study the limit of $\LIM_M$ as
$M\rightarrow \infty$. Under what conditions will $\LIM_M$ converge to
$\FIM$?

\section{PIM and Generalized method of moments}
\label{sec:PIM_achievable}

An efficient unbiased estimator $\what{\pvec}$ exists if and only if the
following identity holds \cite{Kay1993_fundamentals,vanTrees2013_detection}
\begin{equation}\label{eq:efficientcondition}
\der{\ln p_\psub} = \FIM ( \what{\pvec} - \pvec ),
\end{equation}
which is satisfied only in rare cases. In more general scenarios, the maximum
likelihood estimator
\begin{equation}\label{eq:ML}
\what{\pvec} = \argmin_{\pvec} \; -\ln p(\y ; \pvec).
\end{equation}
is (asymptotically) unbiased with (asymptotic)
covariance matrix $\Cov[\what{\pvec}] = \FIM^{-1}$ under certain
regularity conditions. It is thus asymptotically efficient. Given an appropriate initialization
point $\what{\pvec}_{0}$, \eqref{eq:ML} can be solved iteratively using the Newton-based  scoring method:
\begin{equation}\label{eq:scoringmethod}
\what{\pvec}_{i+1} = \what{\pvec}_{i} - \FIM^{-1} \der{-\ln
p_{\psub}}\Bigl|_{\pvec = \what{\pvec}_i}.
\end{equation}

When $\FIM$ and $\der{\ln p_{\psub}}$ are unknown, we may use instead the best misspecified
information matrix \eqref{eq:MFIM} and score
\eqref{eq:misspecifiedscore_basic}, in the sense of
Theorem~\ref{thm:bound}. Then the so-obtained scoring method is
related to a method that minimizes a
certain class of cost functions $V(\pvec)$ which we characterize in what follows. 
Consider evaluating the misspecified variables \eqref{eq:MFIM} and 
\eqref{eq:misspecifiedscore_basic} obtained using
\eqref{eq:tightcondition} and a consistent estimate $\what{\sCov}$
instead of $\sCov$. Then the corresponding scoring method, analogous
to \eqref{eq:scoringmethod}, can be expressed as
\begin{equation}\label{eq:scoringmethod_maxent}
\begin{split}
\what{\pvec}_{i+1} = \what{\pvec}_{i} - \what{\FIM}^{-1}_\star \der{-\ln \what{p}_\star} \Bigl|_{\pvec = \what{\pvec}_i}.,
\end{split}
\end{equation}
where we define (according to the above discussion about \eqref{eq:MFIM}, 
\eqref{eq:misspecifiedscore_basic} and
\eqref{eq:tightcondition})
\begin{equation}\label{eq:V_derivatives}
\begin{split}
\der{\ln \what{p}_\star} &= \dermat^\T(\pvec) \what{\sCov}^{-1} ( \svec - \sMean(\pvec) ) \\
\what{\FIM}_\star&=  \dermat^\T(\pvec) \what{\sCov}^{-1} \dermat(\pvec) .
\end{split}
\end{equation}
One can verify that \eqref{eq:scoringmethod_maxent} is a scoring
method for solving the following problem
\begin{equation}\label{eq:GMM}
\what{\pvec} = \argmin_{\pvec} \; \underbrace{\frac{1}{2} (\svec - \sMean(\pvec) )^\top \what{\sCov}^{-1} (\svec - \sMean(\pvec) )}_{\triangleq V(\pvec)},
\end{equation}
by noting that $\pder_\psub V(\pvec) = -\der{\ln \what{p}_\star}$ and that $\what{\FIM}_\star$ is an estimate of the Hessian $\pder^2_\psub V(\pvec)$. Eq.~\eqref{eq:GMM} is recognized as a generalized method of moments, using an
asymptotically optimal weight matrix $\what{\sCov}^{-1}$
\cite{Hansen1982_large,Soderstrom&Stoica1988_system}.

The cost function $V(\pvec)$ can be characterized around its minimum, as follows
\begin{equation}\label{eq:minimum}
\0 = \pder_\psub V(\what{\pvec}) \simeq \pder_\psub V(\pvec)  +
\pder^2_\psub V(\pvec)(\what{\pvec} - \pvec), 
\end{equation}
where the right-hand side is a Taylor expansion. Using properties of
\eqref{eq:V_derivatives} in \eqref{eq:minimum}, we can solve for $\what{\pvec}$ and
obtain the following approximation
\begin{equation*}
\what{\pvec} \simeq \pvec + (\dermat^\T \what{\sCov}^{-1} \dermat)^{-1}\dermat^\T \what{\sCov}^{-1} ( \svec - \sMean ).
\end{equation*}
Since the unknown distribution $p_{\psub}$ satisfies \eqref{eq:momentconstraint},
it follows that $\E[\what{\pvec}] \simeq \pvec$ and
\begin{equation*}
\Cov[ \what{\pvec} ] \simeq (\dermat^\T \sCov^{-1} \dermat)^{-1} = \LIM^{-1} \succeq \FIM^{-1}.
\end{equation*}
The above expressions hold asymptotically as the number of samples $N$ in
$\y$ increases \cite{Soderstrom&Stoica1988_system}.

In summary, using a scoring method analogous to \eqref{eq:scoringmethod} leads to the generalized method of moments \eqref{eq:GMM} with asymptotic covariance given by the inverse PIM.

\section{Conclusions}

We have provided a direct, algebraic derivation of a tractable lower bound on the
Fisher information matrix which we called the Pearson information
matrix (for reasons explained above). Furthermore, we presented an
information-theoretic link between the PIM and misspecified data
distributions as well as a connection to the generalized
method of moments.

\bibliographystyle{ieeetr}
\bibliography{refs_fimbound}

\end{document}